\documentclass[11pt]{amsart}
\usepackage{amsmath,amssymb,latexsym,esint,cite,mathrsfs, tikz}
\usepackage{verbatim,wasysym}
\usepackage[left=3cm,right=3cm,top=2.7cm,bottom=2.7cm]{geometry}
\usetikzlibrary{decorations.pathreplacing, calligraphy}

\usepackage{microtype}
\usepackage{color,enumitem,graphicx}
\usepackage[colorlinks=true,urlcolor=blue, citecolor=red,linkcolor=blue,
linktocpage,pdfpagelabels, bookmarksnumbered,bookmarksopen]{hyperref}
\usepackage[hyperpageref]{backref}
\usepackage[english]{babel}

\newtheorem{lemma}{Lemma}[section]
\newtheorem{corollary}[lemma]{Corollary} 
\newtheorem{proposition}[lemma]{Proposition}
\newtheorem{theorem}[lemma]{Theorem}
\newtheorem{ex}[lemma]{Example}

\theoremstyle{definition}

\newtheorem{remark}[lemma]{Remark}

\newcommand{\R}{\mathbb{R}}

\renewcommand{\S}{\mathbb{S}}

\newcommand{\K}{\mathcal{K}}

\renewcommand{\H}{\mathcal{H}}
\newcommand{\B}{\mathcal{B}}
\newcommand{\F}{\mathcal{F}}
\newcommand{\A}{\mathcal{A}}
\newcommand{\eps}{\varepsilon}
\newcommand{\be}{\begin{equation}}
\newcommand{\ee}{\end{equation}}

\DeclareMathOperator{\Tr}{tr}
\DeclareSymbolFont{yhlargesymbols}{OMX}{yhex}{m}{n}
\DeclareMathAccent{\wideparen}{\mathord}{yhlargesymbols}{"F3}

\numberwithin{equation}{section}

\title[Uniqueness of the critical point   for solutions of  some $p$-Laplace equations ]{Uniqueness of the critical point for solutions \\ of  some $p$-Laplace equations in the plane }

\author[W.\ Borrelli]{William Borrelli}
\author[S. \ Mosconi]{Sunra Mosconi}
\author[M.\ Squassina]{Marco Squassina}

\address[W.\ Borrelli]{Dipartimento di Matematica e Fisica
	\newline\indent
	Universit\`a Cattolica del Sacro Cuore
	\newline\indent
	Via della Garzetta 48, I-25133 Brescia, Italy}
\email{william.borrelli@unicatt.it}

\address[S.\ Mosconi]{Department of Mathematics and Computer Science
	\newline\indent
	University of Catania
	\newline\indent
	Viale A. Doria 6, I-95125 Catania, Italy}
\email{sunra.mosconi@unict.it}

\address[M.\ Squassina]{Dipartimento di Matematica e Fisica
	\newline\indent
	Universit\`a Cattolica del Sacro Cuore
	\newline\indent
	Via della Garzetta 48, I-25133 Brescia, Italy}
\email{marco.squassina@unicatt.it}


\begin{document}

\begin{abstract}
 We prove that quasi-concave positive solutions to a class of quasi-linear elliptic equations driven by the $p$-laplacian in convex bounded domains of the plane   have only one critical point. As a consequence, we obtain \emph{strict concavity} results for suitable transformations of these solutions.  
 \end{abstract}

\dedicatory{Ad Antonio Ambrosetti, Maestro dell'Analisi Nonlineare, con grande affetto ad ammirazione.}

\maketitle

	\begin{center}
		\begin{minipage}{9cm}
			\small
			\tableofcontents
		\end{minipage}
	\end{center}
 
 \section{Introduction}

\subsection{Overview}
The goal of the present paper is to prove uniqueness of the critical point for solutions of the quasi-linear problem
\begin{equation}\label{equation}
\begin{cases}
	-\Delta_p u=f(u), & \text{in $\Omega$,}   \\
	u>0, & \text{in $\Omega$,}   \\
	u=0, & \text{on $\partial\Omega$\,,}   
\end{cases}
\end{equation}
where $p>1$,  $\Omega\subseteq \R^2$ is a bounded and convex open set and $f$ is a suitable reaction, ensuring that $u$ is actually quasi-concave, meaning that its super-level sets  are convex. 

In order to fix ideas and present the problem, let us suppose for the moment that $f\equiv 1$, so that we are actually looking at the so-called $p$-{\em torsion function}, and $\Omega$ is smooth and strongly convex, but without any symmetry (otherwise other approaches based on \cite{CS} are fruitful). 

For $p=2$ and in the plane, the quasi-concavity of the torsion function $u$ goes back to Makar-Limanov \cite{Makar}, after which many other reactions $f$ have been considered in \cite{ALL, BL, korevaar, kore2, kennington, kaw, CS}. The uniqueness of the critical point for the torsion function of convex domains was first proved via complex functions methods in \cite{haegi} and then reproved in \cite{Makar} and \cite{Gustaffson}; in \cite{APP} the result is obtained via an estimate on the curvature of the level sets. A more fruitful approach was developed by Caffarelli and Friedman in \cite{CaFri}, where they proved that the Hessian of $\sqrt{u}$ is of constant (and thus, full) rank in $\Omega$.  All the previous results have been obtained in the plane.

Caffarelli and Friedman's approach, which is nowadays called {\em Constant Rank Theorem} or {\em Microscopic Convexity Principle},  has then been generalised to arbitrary dimensions in \cite{korlew, bianguan}.  It leads to the uniqueness of the critical point for solutions $u$ of quite general elliptic nonlinear problems of the type
\be
\label{fnl}
\begin{cases}
G(D^2 u, Du, u)=0 &\text{in $\Omega$}\\
u=0&\text{on $\partial\Omega$}
\end{cases}
\ee
 for a strongly convex $\Omega\subseteq \R^N$, via the following route. 

\begin{enumerate}
\item Under suitable convexity-type assumptions on $G$ (see e.\,g.\,\cite{ALL}), there exists an increasing $\varphi$ such that $v=\varphi\circ u$ satisfies a structurally similar elliptic equation and is concave. The critical points for $v$ and $u$ coincide, and are therefore their maximum points.
\item
 By the constant rank principle of \cite{bianguan} (which applies to concave  solutions of \eqref{fnl}),  the Hessian of $v$ has constant rank; the boundary behaviour and the strong convexity of $\Omega$ force  $D^2v$ to have full rank   near $\partial\Omega$, thus everywhere. It follows that $v$ is strictly concave and has a unique maximum point, and so does $u$.
 \end{enumerate}
  As a byproduct of this argument, it turns out that the positive super-level sets of $u$ are strictly convex  and that its maximum point is non-degenerate.
 
This line of proof unfortunately fails for problem \eqref{equation}, even in the model case $f\equiv 1$ and in the plane. The first step still goes through since,  for the solution $u$ of the $p$-torsion problem, the  function $u^{1-1/p}$ is known to be concave by \cite{Saka}. Step two, however, is problematic. The constant rank theorem requires ellipticity of $F$, which lacks for \eqref{equation} precisely at the maximum points of $u^{1-1/p}$. With the available theory, therefore, the best one can prove  is that $u^{1-1/p}$ is strictly convex outside its maximum points, which says nothing about their number. 

Notice that the issue is not a merely technical one. In the unit ball the $p$-torsion function is of the form $u(x)=c\, (1-|x|^{p/(p-1)})$, which is not twice differentiable at the origin if $p>2$. More substantially, for $p<2$, $u$ is actually $C^2$ and $v=u^{1-1/p}$ is concave, but $D^2 v(0)=0$ while  $D^2 v$ has full rank elsewhere, so that the constant rank principle is actually {\em false}. 

Let us finally mention that the $p$-torsion function of a convex domain $\Omega\subseteq \R^N$ has strictly convex super-level sets thanks to \cite{kore3}, but only for levels strictly between $0$ and the maximum of $u$, and this again says nothing about the uniqueness of its critical point.
 
 \subsection{Main result}

 Our approach is in some sense opposite to the one described above. We first prove that the maximum point for solutions $u$ of \eqref{equation} is unique, and then derive the strict concavity of suitable transformations $v=\varphi\circ u$ for an increasing $\varphi$ depending on the reaction $f$. Our main result, for convex bounded domains in the plane, is the following. 

\begin{theorem}
\label{Mth}
Let $f\in  {\rm Lip}_{\rm loc}(\R_+, \R_+)$ be such that $t\mapsto f(t)/t^{p-1}$ is non-increasing on $(0, +\infty)$. If $u\in C^1(\Omega)$ is a quasi-concave solution of \eqref{equation} in a convex bounded $\Omega\subseteq \R^2$, then ${\rm Argmax}(u)$ is a single point.
\end{theorem}

\begin{remark}
Let us make some comments on the assumptions.
\begin{itemize}
\item 
Notice that we are assuming $f(t)>0$ for $t>0$, which ensures that the set  of maximum points has zero measure, thanks to \cite{lou}. This condition  will be assumed in all the manuscript.
\item
The Lipschitz regularity of $f$ is needed to apply some strong comparison principle away from the critical set, proved in \cite{damascelli}.
\item
The assumed monotonicity of $t\mapsto f(t)/t^{p-1}$ ensures the validity of a local weak comparison principle for positive solutions  of  \eqref{equation}, see Lemma \ref{wc}.
\item
The convex body $\Omega$ can have flat parts and corners, i.\,e.\,no strict convexity or regularity (beyond the natural Lipschitz one) is assumed.
\end{itemize}
\end{remark}

A first application of the previous theorem is the following.

\begin{corollary}\label{sclog}
Let $u\in W^{1,p}_0(\Omega)$ solve  \eqref{equation} in a bounded convex domain $\Omega\subset\R^2$ with $C^2$ boundary, where $f\in C^0\big([0, +\infty), [0, +\infty)\big)\cap C^{1, \alpha}_{\rm loc}(\R_+, \R_+)$  for some $0<\alpha<1$ satisfies 
\begin{enumerate}
\item
 $t\mapsto f(t)/t^{p-1}$ is non-increasing on $\R_+$,
\item
$t\mapsto e^{(p-1) t}/f(e^t)$ is convex on $\R$.
\end{enumerate}
Then $\log u$ is strictly concave.
\end{corollary}
 
 \begin{remark}
 \label{r1}\ \vskip1pt
 \begin{itemize}
 \item
 The two required conditions on $f$ ensures that $\log u$ is concave by the results of \cite{BMS}, allowing to apply Theorem \ref{Mth} and deduce strict concavity via additional arguments outlined below and based on the constant rank principle.   
 \item
It is worth underlining that    {\em no strict convexity assumed on $\Omega$}, hence  the super-level sets of $u$ turn out to be  strictly convex, even if $\partial\Omega$ has flat parts.
\item
The regularity of $\partial\Omega$ is required only to ensure that \eqref{equation} has a unique solution under assumption {\em (1)}. Indeed, if this uniqueness property holds true, the approximation argument in \cite[Section 4.1]{BMS} runs through and all the results contained therein follow.  Uniqueness easily holds for the $p$-torsion function in any domain, thus the previous corollary holds true in any  bounded convex $\Omega\subseteq\R^2$.  For example, the $p$-torsion function of a square  has strictly convex positive super-level sets.
More generally, in \cite[Theorem  4.1]{BPZ},  uniqueness for problem \eqref{equation} in any domain has been proved for $f(t)=c\, t^{q-1}$ with $c>0$ and $1\le q<p$, ensuring that for this class of reactions the previous and next corollary hold true without any further assumption on  $\Omega$ beyond convexity and boundedness.
\end{itemize}
 \end{remark}
 
 In a similar manner one can proceed studying strict concavity of more general function $v$ arising as composition of $u$ via suitable transformations. In particular, given a reaction $f$ in \eqref{equation}, we define 
 \[
 F(t)=\int_0^tf(\tau)\, d\tau, 
 \]
 and 
 \be
 \label{defvarphi}
 \varphi(t)=\int_1^t\frac{1}{F^{1/p}(\tau)}\, d\tau.
 \ee
 In \cite{BMS} we studied the concavity of $\varphi(u)$ when $u$ solves \eqref{equation} and the results proved there, together with Theorem \ref{Mth}, provide the following. 
 
\begin{corollary}\label{scvarphi}
Let $u\in W^{1,p}_0(\Omega)$ solve  \eqref{equation} in a bounded convex domain $\Omega\subset\R^2$ with $C^2$ boundary, where $f\in C^0\big([0, +\infty), [0, +\infty)\big)\cap C^{1, \alpha}_{\rm loc}(\R_+, \R_+)$  for some $0<\alpha<1$ satisifes  
\begin{enumerate}
\item
$F^{1/p}$ is concave,
\item
$F/f$ is   convex.
\end{enumerate}
Then $\varphi(u)$ is strictly concave, where $\varphi$ is defined in \eqref{defvarphi}.
\end{corollary}

For a discussion on the relations between the two sets of assumption in the previous corollaries we refer to \cite{BMS}, where also some examples of nonlinearities fulfilling them are given.  Remark \ref{r1} holds for this last statement as well.

\subsection{Sketch of proof}
The proof of Theorem \ref{Mth} relies on Aleksandrov's reflection method. The set ${\rm Argmax}\, (u)$  is a closed convex set with empty interior since $f$ is strictly positive, therefore we must exclude that it is a segment. Arguing by contradiction, we suppose ${\rm Argmax}\, (u)$ is a segment and consider the super-level sets 
\[
\K_\eps=\big\{u> \max_\Omega u-\eps\big\}.
\]
 Our aim is to find a straight cut  of one of the $\K_\eps$'s such that one of the resulting parts  of $\K_\eps$ (called {\em caps} in the following) 
\begin{enumerate}
\item[{\em a)}]
can be reflected around the cut, staying in $\K_\eps$;
\item[{\em b)}]
intersects ${\rm Argmax}(u)$ in a segment of positive length.
 \end{enumerate}
 As long as these two properties are met, the contradiction is found via the strong comparison principle applied to $u$ and its reflection around the cut. The idea to find caps obeying {\em a)} above is by now classical and permits the localisation of various important points related to semilinear problems, (see \cite{BM0} and the literature therein ). It is the simultaneous requirement of {\em a)} and {\em b)} above that is quite tricky to be fulfilled. 
 
 Let's agree to call the caps fulfilling {\em a)} above {\em foldable}, with their {\em width} being the maximum distance of the cap from the cut. Cutting out from a convex set {\em all} its foldable caps, one obtains the so-called {\em heart} of the convex, (see \cite{BM} for some of its properties). 
 
 Back to the proof of {\em a)} and {\em b)}, we first observe that, since $\K_\eps\to {\rm Argmax}\, (u)$ in Hausdorff distance,  {\em b)} is fulfilled as long as $\K_\eps$ has a foldable cap having width uniformly bounded from below by a positive constant, as $\eps\downarrow 0$.  Since $\K_\eps$ converges to a segment, whose heart  is its midpoint, it is reasonable to expect that the heart of $\K_\eps$, as $\eps\downarrow 0$,  will be small compared to its diameter, ensuring the existence  of  foldable cap of large width for sufficiently small $\eps$. Unfortunately, the heart operator is far from being continuous and this argument fails. However, {\em in two dimensions}, any convex set possesses cuts on which the convex set projects itself. We use one of these cuts to construct a foldable cap of $\K_\eps$ with width  comparable to $1/4$ of $\K_\eps$'s diameter (see Lemma \ref{fold} for a precise statement). This provides us with the cap obeying {\em a)} and {\em b)} above, for small $\eps$.
 
 \medskip
 Then, we face an additional difficulty. The strong comparison principle (needed to apply Alexandrov reflection method) for the $p$-Laplacian operator is a delicate matter when the two involved functions have vanishing gradients at the contact points. Indeed, at those point the equation loses ellipticity and the proof of the strong comparison principle relies on quite involved techniques. At present, see \cite{DS}, it is known to hold for the $p$-Laplacian in $\R^2$ (under additional conditions met in our framework)  only for  $p>3/2$.
To deal with the full range $p>1$, we rely of somewhat softer methods, namely
\begin{enumerate}
\item[{\em c)}] the weak comparison principle, ensuring that  $u$ is less than or equal than its reflection around the cut;
\item[{\em d)}]
the strong comparison principle of \cite{damascelli}, under the assumption that the contact point between the compared functions is not critical for both.
\end{enumerate}
The weak comparison principle will do the trick as long as the cut obtained above is not orthogonal to ${\rm Argmax}\, (u)$, since in this case $u$ and its reflection attain the same maximum on different points. Unfortunately, it can actually happen that all foldable caps of $\K_\eps$ intersecting ${\rm Argmax}\, (u)$ arise from cuts orthogonal to the latter. But, if this happens for all $\K_\eps$, the solution turns out  be   one dimensional locally near the cut. In this case $u$ and its reflection coincide on an open set and {\em d)} above  allows to conclude.

\medskip

Let us make one final remark regarding the limits of our proof of Theorem \ref{Mth}. Regardless of the issues related to the strong comparison principle, the main point where the two dimensions play a r\^ole is in finding the cap obeying {\em a)} and {\em b)} above. In the appendix of this manuscript we will construct a sequence of tetrahedra in $\R^3$ converging in the Hausdorff sense to a segment, such that all their foldable caps are disjoint from the limit segment, see Example \ref{exapp}. Therefore, different arguments are likely needed to deal with the corresponding higher dimensional result.

\medskip

The proof of the corollaries  follows as already mentioned in the first point of Remark \ref{r1} when the domain $\Omega$ is strongly convex. A more refined  argument is needed to treat general convex domains and we also want to  avoid any  argument relying on the  regularity of $\partial\Omega$. In this way, whenever uniqueness for problem \eqref{equation} holds true, the statements of the corollaries still stand, as noted in the last point in Remark \ref{r1}.

By the results in \cite{BMS}, a suitable transformation $v=\varphi\circ u$ (with $\varphi$ increasing) is concave, thus $u$ is quasi-concave and Theorem \ref{Mth} applies, giving uniqueness of the critical point $x_{\rm max}\in \Omega$. The constant rank principle applies in  $\Omega\setminus\{x_{\rm max}\}$, where the equation is smooth and elliptic, ensuring that $D^2v$ has constant rank there.  If $\partial\Omega$ is smooth and strongly convex one readily concludes, since $v$ has full rank near $\partial\Omega$ thanks to \cite[Lemma 2.4]{korevaar}. In the general case (which covers arbitrary convex bodies) we proceed by contradiction, assuming that ${\rm det}\, D^2v\equiv 0$  in $\Omega\setminus \{x_{\rm max}\}$. This means that the graph of $v$  is developable there and, by a classical result of Hartman and Nirenberg  \cite{HN},  any point in $\Omega\setminus\{x_{\rm max}\}$ has a segment  going through it on which $Dv$ is constant. In particular, points of arbitrary small gradient can be joined to $\partial\Omega$ through such a segment, leading to  a contradiction since $\partial\Omega$ and $x_{\rm max}$ are at most ${\rm diam}\, (\Omega)$ distance apart. It follows that $D^2v$ is of full rank everywhere in $\Omega\setminus\{x_{\rm max}\}$, implying strict concavity by elementary means.

 \subsection{Outline of the paper}
 In Section \ref{sec2} we collect some preliminary results, stating two comparison principles for solutions to \eqref{equation} and proving some results about convex sets in the plane.  In Section \ref{sec:proof} we prove Theorem \ref{Mth} and Corollary  \ref{scvarphi}, while we omit the proof of Corollary \ref{sclog}  since it follows along the same lines. We conclude with the Appendix where, as already mentioned, we construct a counterexample showing that the applicability of the previously discussed method is limited to the two-dimensional case.
 
 \medskip

{\em Notations:}
In the following $\K$ will always denote a bounded closed convex subset of $\R^N$ and $\omega\in {\S}^{N-1}$ a unit vector in $\R^N$. By ${\mathcal S}_+$ we denote the cone of positive definite $2\times 2$ matrices. By $\langle x, y\rangle$ we denote the scalar product of the vectors $x, y\in \R^N$. The symbol $[x, y]$ stands for the convex envelope of $\{x, y\}$, i.\,e.\,the segment having $x$ and $y$ as extrema. We will write $[x, y]\parallel \omega$ if the line through $x$ and $y$ has direction $\omega$. If $v=(v_1, v_2)$ is a vector in the plane $v^\bot$ will denote any of the vectors $(v_2, -v_1)$ or $(-v_2, v_1)$ orthogonal to $v$ and having same length.

\bigskip 
\medskip
\noindent
{\bf Acknowledgements.}
The authors are members of {\em Gruppo Nazionale per l'Analisi Ma\-te\-ma\-ti\-ca, la Probabilit\`a e le loro Applicazioni} (GNAMPA).
S. Mosconi is partially supported by project PIACERI - Linea 2 and 3 of the University of Catania.
  Part  of  the  paper  was  developed  during  a  visit  of the  second  author  at  the  Department  of  Mathematics  and  Physics of the Catholic University  of Sacred Hearth, Brescia, Italy.  
  The  hosting  institution  is  gratefully  acknowledged.

 \section{Preliminaries}\label{sec2}
 \subsection{Comparison principles}
 The following comparison principle is essentially contained in \cite{diaz-saa}. If both compared functions are positive on $\overline\Omega$ (which suffices for our purposes), its proof is particularly simple and we provide it for completeness.
 
 \begin{lemma}\label{wc}
 Let $f\in C^0(\R_+)$ and suppose that $u_1, u_2\in C^1(\Omega)\cap C^0(\overline\Omega)$ are positive in $\overline\Omega$ and solve\eqref{equation}
 in $\Omega$ such that $u_1\ge u_2>0$   on $\partial\Omega$. If $t\mapsto f(t)/t^{p-1}$ is non-increasing on $\R_+$, then $u_1\ge u_2$ in $\Omega$.
 \end{lemma}
 
 \begin{proof}
 Suppose by contradiction that there exists a  nonempty connected component $\Omega_0$ of $\{x\in \Omega:u_2(x)>u_1(x)\}$.  By the continuity of the $u_i$ in $\overline\Omega$ and and the assumption $u_1\ge u_2$ on $\partial\Omega$, it holds $u_1=u_2$ on $\partial\Omega_0$.
 
 Recall that the following Picone inequality
 \begin{equation}
\label{picone}
|\nabla v|^{p-2}\nabla v\cdot \nabla \frac{w^p}{v^{p-1}}\le |\nabla w|^p,
\end{equation}
valid for any positive $v, w\in C^1$, becomes an equality in a connected set if and only if $v=k\, w$, with $k>0$. 
Using \eqref{picone} for $v=u_i$ and $w=u_j$ for $i\ne j$, we get
\be
\label{piconei}
|\nabla u_i|^{p-2}\nabla u_i\cdot \nabla \frac{u_j^p}{u_i^{p-1}}\le |\nabla u_j|^p=|\nabla u_j|^{p-2}\nabla u_j\cdot \nabla \frac{u_j^p}{u_j^{p-1}}.
\ee
 We sum the previous two inequalities and rearrange to get
\be
\label{kqp0}
|\nabla u_1|^{p-2}\nabla u_1\cdot \nabla \frac{\varphi}{u_1^{p-1}}\le |\nabla u_2|^{p-2}\nabla u_2\cdot \nabla \frac{\varphi}{u_2^{p-1}}
\ee
where $\varphi=(u_2^p-u_1^p)_+$.
We integrate over $\Omega_0$ and notice that $\varphi/u_i^{p-1}\in W^{1,p}_0(\Omega_0)$, thanks to the positivity and regularity of the $u_i$'s. Using  equation \eqref{equation}, we get
\[
\begin{split}
\int_{\Omega_0}\frac{f(u_1)}{u_1^{p-1}}\, \big(u_2^{p}-u_1^{p}\big)\, dx& =\int_{\Omega_0}|\nabla u_1|^{p-2}\nabla u_1\cdot \nabla \frac{\varphi}{u_1^{p-1}}\, dx\\
&\le \int_{\Omega_0} |\nabla u_2|^{p-2}\nabla u_2\cdot \nabla \frac{\varphi}{u_2^{p-1}}\, dx= \int_{\Omega_0}\frac{f(u_2)}{u_2^{p-1}}\, \big(u_2^{p}-u_1^{p}\big)\, dx
\end{split}
\]
so that
\be
\label{kqp}
\int_{\Omega_0}\left(\frac{f(u_1)}{u_1^{p-1}}-\frac{f(u_2)}{u_2^{p-1}}\right) \big(u_2^{p}-u_1^{p}\big)\, dx \le 0
\ee
The monotonicity assumption on $f$ ensures that the integrand is non-negative, hence it vanishes identically in $\Omega_0$. It follows that none of the inequalities in \eqref{piconei} can be strict on an open subset of $\Omega_0$, for otherwise the inequality in  \eqref{kqp0} would be strict there and the left hand side of \eqref{kqp} would be negative. Thus equality is attained in \eqref{piconei}, forcing $u_2=k\, u_1$ in $\Omega_0$. Since $u_2>u_1>0$ in $\Omega_0$,  we have  $k>1$, and since the $u_i$ are positive in $\overline\Omega$, it follows that $u_2>u_1$ on $\partial\Omega_0$. This contradiction implies $\Omega_0=\emptyset$, proving the claim.
 \end{proof}

 The following strong comparison principle is taken from \cite[Theorem 1.4 and Remark 1.4]{damascelli}.
 
 \begin{proposition}\label{sc}
 Suppose that $u_1, u_2\in C^1(\Omega)$ solve \eqref{equation} with $f$ being Lipschitz on the image of $u_1$ and $u_2$. Suppose that $u_1\le u_2$ in $\Omega$ and $u_1(x_0)=u_2(x_0)$ for some $x_0\in \Omega\setminus Z$, where 
 \[
 Z=\big\{x\in \Omega:\nabla u_1(x)=\nabla u_2(x)=0\big\}.
 \]
 Then $u_1=u_2$ in the connected component of $\Omega\setminus Z$ containing $x_0$.
 \end{proposition}
 
 \subsection{Convex geometry}

 Given a  convex $\K\subseteq \R^N$,  its {\em support function} $\H_\K:\S^{N-1}\to \R$ is 
\[
\H_\K(\omega)=\sup\big\{\langle x, \omega\rangle:x\in \K\big\}.
\]
 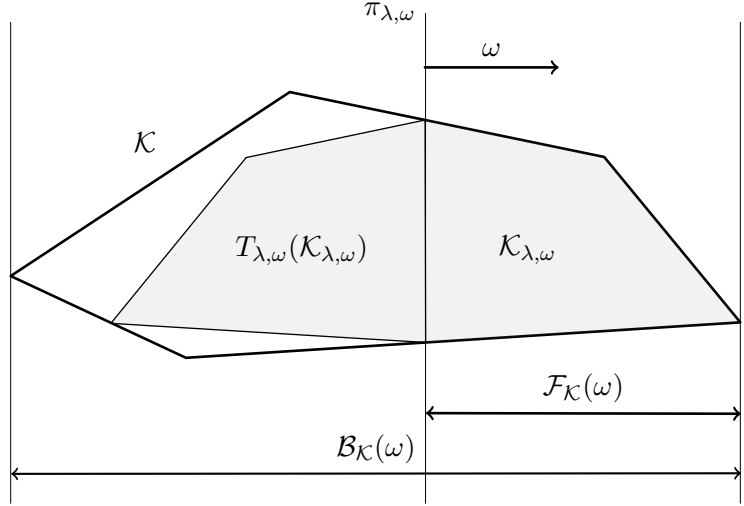
\begin{figure}[h]
 \begin{tikzpicture}[line cap=round,line join=round,x=0.5cm,y=0.4cm]
\draw[line width=1pt] (-8.57,-0.44) -- (-1.15,5.68) -- (7.21,3.52) -- (10.83,-1.98) -- (-3.91,-3.16) -- cycle;
\fill[fill=gray,fill opacity=0.1] (-5.907098987259063,-2.013407383208102) -- (-2.3090838320013387,3.5009998326706557) -- (2.447983278390904,4.750377526157375) -- (2.4627541592651974,-2.6498337918634376) -- cycle;
\fill[fill=gray,fill opacity=0.1] (2.447983278390904,4.750377526157375) -- (7.21,3.52) -- (10.83,-1.98) -- (2.4627541592651974,-2.6498337918634376) -- cycle;
\draw[very thin]  (-8.57, -8)--(-8.57, 8)  (10.83, -8)--(10.83, 8)  (2.46, -8)--(2.46, 8.3) node[left]{$\pi_{\lambda, \omega}$} ;
\draw (-5, 4) node{$\mathcal{K}$};
\draw [line width=0.5pt] (-5.907098987259063,-2.013407383208102)-- (-2.3090838320013387,3.5009998326706557);
\draw [line width=0.5pt] (-2.3090838320013387,3.5009998326706557)-- (2.447983278390904,4.750377526157375);
\draw [line width=0.5pt] (2.447983278390904,4.750377526157375)-- (2.4627541592651974,-2.6498337918634376);
\draw [line width=0.5pt] (2.4627541592651974,-2.6498337918634376)-- (-5.907098987259063,-2.013407383208102);
\draw [line width=1pt, ->] (2.443493438299297,6.5 )-- node[midway, above]{$\omega$} (6,6.5);
\draw [line width=1pt, <->] (2.4674451199592036,-5.000005099560959)-- node[midway, above]{$\mathcal{F}_{\mathcal{K}}(\omega)$}(10.847564879961116,-5.020004860519039);
\draw [thick, <->] (-8.556866160429001,-7.020053625070716)-- node[midway, above]{$\mathcal{B}_{\mathcal{K}}(\omega)$} (10.851517557629023,-7.000296372140461) ;
\draw (5.17,0.43) node {$\mathcal{K}_{\lambda, \omega}$};
\draw[color=black] (-0.83,0.43) node {$T_{\lambda, \omega}(\mathcal{K}_{\lambda, \omega})$};
 
\end{tikzpicture}
\caption{Some quantities defined for a convex set}
\end{figure}
If  $\K$ is bounded, $\H_\K$ turns out to be continuous. The {\em breadth } $\B_{\K}:\S^{N-1}\to \R$ of $\K$ is
\[
\B_\K(\omega)=\H_\K(\omega)+\H_\K(-\omega)
\]
i.\,e.\, it is the minimal distance between two parallel supporting hyperplanes of $\K$ having normal vector $\omega$.  For any $\omega\in \S^{N-1}$ both $\K\mapsto \H_{\mathcal K}(\omega)$ and $\K\mapsto \B_\K(\omega)$ are continuous with respect to Hausdorff convergence  in the class of convex subsets of a bounded set.
The {\em width} of $\K$ is  its minimal breadth, i.\,e.
 \[
{\rm width}\, (\K)=\inf\left\{\B_\K(\omega):\omega\in \S^{N-1}\right\}\, .
 \]
A {\em section} of $\K$ is  its intersection with an hyperplane; a {\em shadow of $\K$ in direction $\omega$} is the image of $\K$ under an orthogonal projection on a hyperplane having normal vector $\omega\in \S^{N-1}$. Given an hyperplane with equation
\[
\pi_{\lambda, \omega}=\left\{x\in \R^N:\langle x, \omega\rangle=\lambda\right\},
\]  
we denote by  $T_{\lambda, \omega}$ the reflection on  $\pi_{\lambda, \omega}$, i.\,e.
\be
\label{reflection}
T_{\lambda, \omega}(x)=x-2\, \omega\, \big( \langle \omega, x\rangle- \lambda\big)
\ee
 and we define the corresponding {\em cap} of $\K$ as
\[
\K_{\lambda, \omega}=\left\{x\in \K:\langle x, \omega\rangle\ge \lambda\right\},
\]
i.\,e.\,$\K_{\lambda, \omega}$ is the part of $\K$ above $\pi_{\lambda, \omega}$ (in the direction $\omega$). The {\em maximal folding cap} of $\K$ in direction $\omega$ is defined as 
\[
\K_\omega=\bigcup\left\{\K_{\lambda, \omega}: T_{\lambda, \omega}(\K_{\lambda, \omega})\subseteq \K\right\}.
\]
 Finally, the {\em maximal folding height} $\F_\K:{\mathbb S}^{N-1}\to [0, +\infty)$ of $\K$ is 
\[
\F_\K(\omega)= \B_{\K_\omega}(\omega).
\]
By \cite[Lemma 2.1]{BM} and the representation
\[
\F_\K(\omega)=\H_{\mathcal K}(\omega)-\min\{\lambda\in \R:T_{\lambda, \omega}(\K_{\lambda, \omega})\subseteq \K\}
\]
the maximal folding height is upper semicontinuous.
  
  \begin{lemma}\label{fold}
Let $\K\subseteq\R^N$ be convex and such that $\pi_{\lambda, \omega}\cap \K$  is a shadow of $\K$. Then 
  \be
  \label{eqlem1}
\max\left\{\F_\K(\omega), \F_\K(-\omega)\right\}\ge \frac{1}{4} \, \B_\K(\omega)\, .
  \ee
  In particular, if 
  \be
  \label{media}
  \lambda\le \frac{\H_\K(\omega)-\H_\K(-\omega)}{2}, 
  \ee
  then  
  \[
  T_{\mu, \omega}(\K_{\mu, \omega})\subseteq \K\quad \text{ for}\quad  \mu:= \H_\K(\omega)-\frac{1}{4}\, \B_\K(\omega)\ge   \lambda+\frac{1}{4}\, \B_{\K}(\omega).
  \]
  \end{lemma}
  
  \begin{proof}
  Let   $\Pi_{\lambda, \omega}$ be the orthogonal projection onto $\pi_{\lambda, \omega}$, so that $\Pi_{\lambda, \omega}(\K)=\pi_{\lambda, \omega}\cap \K$. 
It must hold
 \[
 \H_\K(\omega)\ge \lambda\ge -\H_\K(-\omega)
 \]
  and we define  
  \[
 \lambda_1=\frac{\H_\K(\omega)+\lambda}{2}, \qquad  \lambda_2=\frac{\H_\K(-\omega)-\lambda}{2}\,. 
  \]
For  the caps $ \K_{ \lambda_1, \omega}$ and $\K_{-\lambda_2, -\omega}$
 we claim that 
 \be
 \label{claimrefl}
 T_{\lambda_1, \omega}( \K_{ \lambda_1, \omega})\subseteq \K, \qquad  T_{-\lambda_2, - \omega}( \K_{-\lambda_2, -\omega})\subseteq \K.
 \ee
  By convexity, given any point $x\in \K$, the segment  $[x, \Pi_{\lambda, \omega}(x)]$ is contained in $\K$. If $x\in  \K_{ \lambda_1, \omega}$, by \eqref{reflection}
\[
  \langle\omega,  T_{\lambda_1, \omega}(x)\rangle=2\,  \lambda_1-\langle\omega, x\rangle\ge 2\,  \lambda_1-\H_\K(\omega)= \lambda\, , 
\]
and we infer that $T_{\lambda_1, \omega}(x)$ lies on the segment $[x, \Pi_{\lambda, \omega}(x)]$ and, a fortiori, in $\K$. A symmetric argument shows that for any $x\in \K_{-\lambda_2, -\omega}$, the point $T_{-\lambda_2, -\omega}(x)$ lies on the segment $[x, \Pi_{\lambda,\omega}(x)]$, thus proving \eqref{claimrefl}.
It is readily checked that 
\[
\B_{\K_{  \lambda_1, \omega}}(\omega)=\frac{\H_\K(\omega)-\lambda}{2}, \qquad \B_{\K_{- \lambda_2, -\omega}}(\omega)=\frac{ \H_\K(-\omega)+\lambda}{2}
\]
so that 
\[
\begin{split}
\max\left\{\F_\K(\omega), \F_\K(-\omega)\right\}&\ge \max\big\{\B_{  \K_{  \lambda_1, \omega}}(\omega), \B_{  \K_{  - \lambda_2, -\omega}}(\omega)\big\}\\
&= \frac{1}{2}\, \max \left\{\H_\K(\omega)-\lambda,  \H_\K(-\omega)+\lambda\right\}\\
&\ge \frac{\H_\K(\omega)+\H_\K(-\omega)}{4}
\end{split}
\]
as claimed. The second assertion follows from \eqref{claimrefl} and the fact that, under assumption \eqref{media}, it holds 
\[
  \lambda_1\le \H_\K(\omega)-  \frac{\H_\K(\omega)+\H_\K(-\omega)}{4}=  \H_\K(\omega)-\frac{1}{4}\, \B_\K(\omega)\,.
\]
  \end{proof}

  The factor $1/4$ in \eqref{eqlem1} is optimal, as one can check considering a parallelogram constructed joining  a pair of congruent right isoscele triangles through a short side, see the following figure.
  
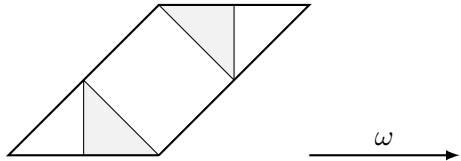
\begin{figure}[h]
  \label{figoptimal}
 \begin{center}
 \begin{tikzpicture}[scale=2]
 \draw[thick] (0, 0)--(1, 1)--(0, 1)--(-1, 0)--(0, 0);
 \fill[fill=gray, fill opacity=0.1]  (0.5, 0.5)--(0.5, 1)--(0, 1)--(0.5, 0.5);
 \fill[fill=gray, fill opacity=0.1]  (-0.5, 0.5)--(0, 0)--(-0.5, 0)--(-0.5, 0.5);
  \draw  (0.5, 0.5)--(0.5, 1)--(0, 1)--(0.5, 0.5);
   \draw (-0.5, 0.5)--(0, 0)--(-0.5, 0)--(-0.5, 0.5);
   \draw[thick, -latex] (1, 0)--node[midway, above]{$\omega$} (2, 0);
 \end{tikzpicture}
 \end{center}
 \end{figure}
  \begin{lemma}\label{section}
Any convex body $\K\subseteq \R^2$ has a section which is a shadow in  direction $\omega^\bot$, where $\omega$ is a direction of minimal breadth.
\end{lemma}

  \begin{proof}
  This follows from a well known characterisation of the width of a convex body in $\R^N$.  By   \cite[Section 33]{BF},   it holds
  \[
{\rm width}\, (\K)= \min_{\omega\in \S^{N-1}}  \max\left\{|x_1- x_2|:x_1, x_2\in \K, [x_1, x_2]\parallel \omega\right\}.
  \]
  and the right hand side is attained at some $\bar\omega$ and $x_1, x_2\in \partial \K$ such that there are two parallel supporting hyperplanes of $\K$ through $x_1$ and $x_2$. If $\omega$ is the normal to these hyperplanes such that $\langle \omega, x_1-x_2\rangle\ge 0$, it holds  by construction  $\B_{\K}(\omega)=\langle \omega, x_1-x_2\rangle$ and, by Schwartz inequality and the definition of width,
  \[
  {\rm width}\, (\K)\le \B_{\K}(\omega)\le |x_1-x_2|={\rm width}\, (\K).
  \]
Therefore $\langle \omega, x_1-x_2\rangle=|x_1-x_2|$, implying that $\omega$ and $x_1-x_2$ are proportional and thus, being $[x_1-x_2]\parallel\bar\omega$, that $\omega=\pm\bar\omega$. In particular, $\bar\omega$ is a direction of minimal breadth and $\K$ lies between two hyperplanes orthogonal to $[x_1, x_2]\subseteq \K$, passing through $x_1$ and $x_2$. In two dimensions, this is equivalent to say that the section $[x_1, x_2]$ is a shadow in direction $\bar\omega^\bot$.
  \end{proof}
  
As already pointed out, the maximal folding height is only upper semicontinuous. The failure of continuity is due to flat parts of the boundary of $\K$, but the maximal folding height is too much nonlocal to determine where the flat parts are located when lower semicontinuity fails.  A way to localise the maximal folding height is to consider the function $\omega\mapsto \F_{\K_{\bar\lambda, \bar\omega}}(\omega)$
for suitable fixed caps $\K_{\bar\lambda, \bar\omega}$ of $\K$ and study its behaviour as $\omega\to \bar\omega$. Notice that $\F_\K(\bar\omega)\ge \F_{\K_{\bar\lambda, \bar\omega}}(\bar\omega)$ for any $\bar\lambda\in \R$, $\bar\omega\in \S^{1}$.
 
\begin{lemma}\label{rectangle}
Let $\K\subseteq \R^2$ be a convex body and $\pi_{\bar\lambda, \bar\omega}\cap \K$ be a  shadow  of $\K$. If 
\be
\label{liminf}
\liminf_{\omega\to \bar\omega} \F_{\K_{\bar\lambda, \bar \omega}}(\omega)<\F_{\K_{\bar\lambda, \bar\omega}}(\bar\omega)
\ee
then 
\[
\K\cap \Big\{x\in \R^2: \bar\lambda\le \langle x, \bar\omega\rangle\le \frac{\bar\lambda+\H_\K(\bar\omega)}{2}\Big\}
\]
 is a rectangle with sides parallel to $\bar\omega$ and $\bar\omega^\bot$.
\end{lemma}

 \begin{proof}
 We will use the standard Euclidean geometry notation and the slope of a segment or a line is henceforth defined by measuring the angle formed by the latters with respect to horizontal lines in a fixed orthogonal coordinate system. We can suppose that $\K=\K_{\bar\lambda, \bar\omega}$ and,
after a rotation and translation, that $\bar\omega=(1, 0)$ and $\bar\lambda=0$, so that $\K$ is contained in a minimal rectangle $ABCD$ with sides parallel to the coordinate axes, with the segment $AB$ being a shadow of $\K$. In other terms 
\[
AB\subseteq \K\subseteq ABCD
\]
where, here and henceforth, by a sequence of points we understand their convex envelope.
\begin{figure}[b]
\centering

\begin{tikzpicture}[line cap=round,line join=round,scale=0.66]

\clip(-15,1) rectangle (-0.5,9);
\draw [line width=1pt] (-14,8)-- (-14,2);
\draw [line width=1pt] (-14,2)-- (-2,2);
\draw [line width=1pt] (-2,2)-- (-2,8);
\draw [line width=1pt] (-2,8)-- (-14,8);
\draw [line width=1pt] (-8,2)-- (-8,8);
\draw [line width=1pt,domain=-17.806666666666665:8.02] plot(\x,{(--44.84--2.98*\x)/6});
\draw [line width=1pt] (-14,7.3)-- (-2,7.3);
\draw [line width=1pt] (-14,2)-- (-8,3.5);

\draw [fill=black] (-14,8) node[left]{$A$} circle (1.5pt);
\draw [fill=black] (-14,2) node[left]{$B$} circle (1.5pt);
\draw [fill=black] (-2,2) node[right]{$C$} circle (1.5pt);
\draw [fill=black] (-2,8) node[right]{$D$} circle (1.5pt);
\draw [fill=black] (-8,2) node[below]{$M$} circle (1.5pt);
\draw [fill=black] (-8,8) node[above]{$N$} circle (1.5pt);
\draw [fill=black] (-8,3.5) node[below right]{$Q$} circle (1.5pt);
\draw [fill=black] (-2,6.48) node[below right]{$E$} circle (1.5pt);
\draw [fill=black] (-8,7.3) node[above right]{$P$} circle (1.5pt);
\draw [fill=black] (-2,7.3) node[right]{$G$}circle (1.5pt);
\draw [fill=black] (-14,7.3) node[left]{$F$} circle (1.5pt);
 
\end{tikzpicture}
\caption{The minimal rectangle $ABCD$ containing $\K$.}
\label{figgeo}
\end{figure}
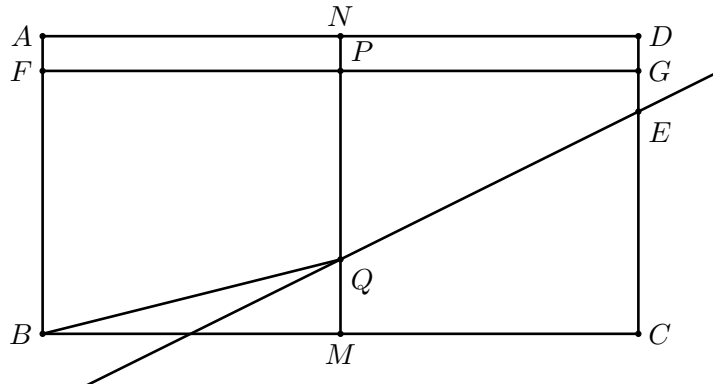
Let $M$ be the midpoint of $BC$ and $N$ the midpoint of $AD$, so that we have to prove that $\K\cap ABMN=ABMN$. By convexity, it suffices to prove that $M$ and $N$ belong to $\K$. Let $PQ=\K\cap MN$, with $P$ nearest to $N$ and $Q$ nearest to $M$ and assume, by contradiction and without loss of generality,  that $Q\ne M$. Let $r$ be a support line for $\K$ at $Q$. Since $ABCD$ is the smallest rectangle containing $\K$ and $\K$ is closed, $\K\cap CD$ is not empty. We deduce by convexity that $r$ has positive slope (since it cannot intersect $AB$), and must intersect $CD$ at a point $E\ne C$. Pick the horizontal line through $P$ and let $F$, $G$ be its intersections with  $AB$ and $CD$ respectively. The slope of a support line for $\K$ at $P$ must be non-positive, so that we deduce, (see Figure \ref{figgeo})
\be
\label{g1}
\K\cap MNDC\subseteq PQEG\,.
\ee
Moreover, the segments $BQ$ and $PF$ lie in $\K$ by convexity, hence
\be
\label{g2}
PQBF\subseteq \K\,.
\ee

Let ${\mathcal C}$ be the circle of center $P$ and radius $|PE|$ and consider the arc $\partial{\mathcal C}\cap PQBF$: it is non-empty since, being $P$ the midpoint of $FG$,  one of its extrema  is   the symmetric of $E$ under a reflection over the line through $P$ and $Q$. Let $R$ be such point. The arc $\partial{\mathcal C}\cap PQBF$ has   positive length, otherwise $\partial{\mathcal C}$ would be externally tangent to $PQBF$,  forcing $R=B$, thus $C=E$ and finally $Q=M$, contrary to the assumption.
Hence $\partial{\mathcal C}\cap PQBF=\wideparen{RS}$ for some point $S\ne R$. 

We have  $|PE|\ge |PG|$ since $PGE$ is right in $G$ and from $|PG|=|PF|$, we deduce that $PF\setminus\{R\}$ is interior to ${\mathcal C}$. In particular, $S\notin PF$. Since $R\in BF$, the point $S$ cannot belong to $BF$ either, thus $S\in BQ\cup PQ$. If $S\in BQ$, since ${\rm slope}\,(QE)\ge {\rm slope}\,(BQ)$, it holds
\be
\label{slope}
{\rm slope}\,(SE)\le {\rm slope}\,(QE)\,.
\ee
The same is trivially verified if $S\in PQ$, so the previous display always holds.

The axis of $SE$ goes through $P$, forming an angle $\alpha_0$ with $PQ$. Notice that  $S\ne R$ implies that $\alpha_0>0$.  
 Let $s$ be a line through $P$ and  a point  $H_s\in QE$ such that 
\be
\label{angle}
0<Q\widehat{P}H_s<\alpha_0
\ee
and $T_s$ be the reflection around $s$.

\begin{figure}[b]
\centering
 
\begin{tikzpicture}[line cap=round,line join=round, scale=0.65]
\clip(-14.6,-0.1) rectangle (-1.3,9);
\draw [line width=1pt] (-14,2)-- (-2,2);
\draw [line width=1pt] (-14,7.36)-- (-2,7.36);
\draw [line width=1pt] (-14,2)-- (-8,2.466666666666668);
\draw [line width=0.5pt] (-8,7.36) circle (6.337318044725228cm);
\draw [line width=1pt,domain=-17.806666666666665:2.9] plot(\x,{(-45.3696-7.026666666666664*\x)/1.4733333333333407});
\draw [line width=0.5pt, dotted] (-12.675185646648384,3.0816546224769414)-- (-2,5.32);
\draw [line width=0.5pt, dotted] (-13.5,5.05)-- (-2,7.36);
\draw [line width=1pt] (-14,7.36)-- (-14,2);
\draw [line width=1pt] (-2,7.36)-- (-2,2);
\draw [line width=1pt] (-8,7.36)-- (-8,2.466666666666668);
\draw [line width=1pt] (-8,2.466666666666668)-- (-2,5.32);
\draw [fill=black] (-13.5,5.05)  circle (1.5pt);
\draw  (-13.18,5.09) node[above]{$T_s(G)$} ;
\draw [fill=black] (-14,2) node[left]{$B$} circle (1.5pt);
\draw [fill=black] (-2,2) node[right]{$C$} circle (1.5pt);
\draw [fill=black] (-8,2.466666666666668) circle (1.5pt);
\draw (-8.1,2.63) node[below right]{$Q$};
\draw [fill=black] (-2,5.32) node[right]{$E$} circle (1.5pt);
\draw [fill=black] (-8,7.36) node[above right]{$P$} circle (1.5pt);
\draw [fill=black] (-2,7.36) node[above]{$G$} circle (1.5pt); 
\draw [fill=black] (-14,7.36) node[above]{$F$} circle (1.5pt);
\draw [fill=black] (-11.654373757474964,2.1824375966408365)  circle (1.5pt);
\draw  (-11.6,2.19) node[above]{$S$};
\draw [fill=black] (-14,5.32) node[left]{$R$} circle (1.5pt);
\draw[color=black] (-6.1,0.1) node {$s$};
\draw [fill=black] (-12.675185646648384,3.0816546224769414) circle (1.5pt);
\draw[color=black] (-12.38,3.5) node {$T_s(E)$};
\draw [fill=black] (-7.067009973274226,2.9103552571540345)  circle (1.5pt);
\draw (-7.08,3.1) node[below right]{$H_s$};
 
\end{tikzpicture}

 \label{figgeo2}
 \caption{}
\end{figure}

\noindent
If $x$ is a point of $\K$ on the right of $s$, let    $s^\bot_x$  be the perpendicular to $s$ passing through $x$.  Thus $s^\bot_x$ intersects the convex polygon 
\[
{\mathcal P}=FGEQB
\]
 in two points. From \eqref{angle} and \eqref{slope} we infer
\[
0<{\rm slope}\,(s^\bot_x)<{\rm slope}\,(SE)\le {\rm slope}\, (QE)\, ,
\]
hence  the leftmost intersection of $s^\bot_x$ with ${\mathcal P}$,  denoted by $x_s$, belongs to $FB\cup BQ$, which is contained in $\K$ by \eqref{g2}.
Therefore $[x, x_s]$ is contained in  $\K$. 

We claim that 
\be
\label{claim1}
T_s(PGEH_s)\subseteq {\mathcal P}\,.
\ee
To prove it, recall that $P, H_s\in {\mathcal P}$ so that it suffices by convexity to check that $T_s(E)$ and $T_s(G)$ belong to ${\mathcal P}$. This is clear for $T_s(E)$ since, as $Q\widehat{P}H_s$ varies between $0$ and $\alpha_0$, $T_s(E)$ runs over the arc $\wideparen{RS}\subseteq {\mathcal P}$. On the other hand $T_s(G)$ lies below the line through $FG$ and to the left of $s$ by construction. Moreover, it lies  on the circle with center $P$ and radius $|PG|$, which is tangent to $AB$ in $F$, therefore it lies on the right of the line through $FB$. Finally,   $T_s(G)G$ is parallel to $T_s(E)E$ and thus $T_s(G)$ lies above the line through $T_s(E) E$. It follows that    $ T_s(G)\in{\mathcal P}$, concluding the proof of \eqref{claim1}.

We can now show that  the reflection $T_s$ around $s$ of the cap $\K_s$ of $\K$ on the right of $s$ lies in $\K$. Indeed, \eqref{g1} implies that $ \K_s\subset PGEH_s$ and, given $x\in \K_s$,  \eqref{claim1} forces $T_s(x)$ to belong to $[x, x_s]\subseteq \K$. In particular, if $\omega_s\in \S^{N-1}$ is the normal to $s$ in the positive $x$-direction, 
\be
\label{dis}
\F_\K(\omega_s)\ge {\rm dist}\, (s\cap ABCD, CD)
\ee
since $CD$ contains at least a point of $\K$. On the other hand  $\F_\K(\bar\omega)=|AB|/2$
 and it is readily verified that, as $s$ converges to the line through $PQ$, the right hand side of \eqref{dis} converges to $|AB|/2$, contradicting \eqref{liminf}.
\end{proof}

  \section{Proof of the main Theorem and its consequences}\label{sec:proof}
  
\begin{proof}[Proof of Theorem \ref{Mth}] Let $\A={\rm Argmax}(u)$ which, by \cite{lou}, must have Lebesgue measure  zero thanks to the positivity of $f$. By the quasi-concavity assumption on $u$, $\A$ is convex, so that it can be either a point or a segment. Hence, in order to prove the theorem, we show that the latter case cannot occur, arguing by contradiction.
To this aim, suppose that $\A$ is a segment of length $\ell$ and with midpoint $x_0$, and denote by $\omega_\A\in \S^{1}$ a unit vector parallel to $\A$. 

 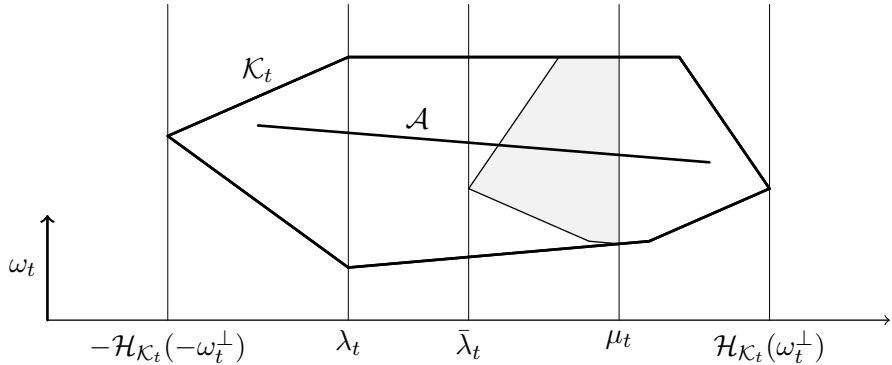
\begin{figure}[b]
\centering
\begin{tikzpicture}[line cap=round,line join=round, ,x=0.4cm,y=0.35cm]
 
\draw[line width=1pt] (7,0) -- (13,3) -- (24,3) -- (27,-2) -- (23,-4) -- (13,-5) -- cycle;
\fill[line width=1pt,fill=gray,fill opacity=0.1] (22,3) -- (20,3) -- (17,-2) -- (21,-4) -- (22,-4.1) -- cycle;
\draw [line width=1pt] (7,0)-- (13,3);
\draw [line width=1pt] (13,3)-- (24,3);
\draw [line width=1pt] (24,3)-- (27,-2);
\draw [line width=1pt] (27,-2)-- (23,-4);
\draw [line width=1pt] (23,-4)-- (13,-5);
\draw [line width=1pt] (13,-5)-- (7,0);
\draw [line width=0.3pt] (13,5)-- (13,-7);
\draw [line width=0.3pt] (7,5)-- (7,-7);
\draw [line width=0.3pt] (27,5)-- (27,-7);
\draw [line width=0.3pt] (22,5)-- (22,-7);
\draw [line width=1pt] (10,0.4)-- node[pos=0.35, sloped, above]{$\mathcal{A}$}(25,-1);
\draw [line width=0.3pt] (17,5)-- (17,-7);

\draw [line width=0.5pt, ->] (3,-7)-- (31,-7);

\draw [line width=1pt, ->] (3,-7)-- node[left, midway]{$\omega_t$}(3,-3) ;

\draw [line width=0.5pt] (22,3)-- (20,3);
\draw [line width=0.5pt] (20,3)-- (17,-2);
\draw [line width=0.5pt] (17,-2)-- (21,-4);
\draw [line width=0.5pt] (21,-4)-- (22,-4.1);
 \draw (10, 2.5) node{$\mathcal{K}_t$};
\draw (7,-7) node[below]{$-\mathcal{H}_{\mathcal{K}_t}(-\omega_t^\bot)$};
\draw  (13,-7) node[below]{$\lambda_t$};
\draw (17,-7) node[below]{$\bar\lambda_t$};
\draw  (22,-7) node[below]{$\mu_t$};
\draw   (27,-7) node[below]{$\mathcal{H}_{\mathcal{K}_t}(\omega_t^\bot)$};
\end{tikzpicture}
\caption{The convex body $\K_t$ and some related objects}
\end{figure}

For any  level $t<M:=\max u(\Omega)$, consider the convex body  $\K_t=\{x\in \Omega:u(x)\ge t\}$ and let $\omega_t$ be the direction of minimal breadth of $\K_t$. By Lemma \ref{section}, there exists $\lambda_t$ such that 
 \be
 \label{barlambda}
 \pi_{\lambda_t, \omega_t^\bot}\cap \K_t\quad \text{is a shadow of $\K_t$}
\ee
and we can fix the orientation of $\omega^\bot_t$ in such a way that
\be
\label{condlambda}
\lambda_t\le \bar\lambda_t:=\frac{\H_{\K_t}(\omega^\bot_t)-\H_{\K_t}(-\omega^\bot_t)}{2}\,.
\ee
Let 
\be
\label{mut}
\mu_t:= \H_{\K_t}(\omega_t^\bot)-\frac{1}{4}\, \B_{\K_t}(\omega_t^\bot)\ge\lambda_t+\frac{1}{4}\, \B_{\K_t}(\omega_t^\bot)
\ee
and 
\[
 \K_{t, \mu_t, \omega_t^\bot}=\{x\in \K_t: \langle x, \omega_t^\bot\rangle\ge \mu_t\}\,,
 \]
so that Lemma \ref{fold} ensures 
\be
\label{incl}
T_{\mu_t, \omega_t^\bot}(\K_{t, \mu_t, \omega_t^\bot})\subseteq \K_t\,.
\ee

  As $t\uparrow M$, $\K_t$ converges to $\A$ with respect to the Hausdorff distance, therefore 
 \[
 \lim_{t\uparrow M}\B_{\K_t}(\omega_t)=0, \qquad \lim_{t\uparrow M}\B_{\K_t}(\omega^\bot_t)=\ell\,,
 \]
where $\ell>0$ is the length of $\A$. Hence there exists $\bar t$ such that  if $M>t>\bar t$, it holds 
 \be
 \label{condt}
 \B_{\K_t}(\omega_t)\le \alpha\, \ell \qquad  \B_{\K_t}(\omega^\bot_t)\le \beta\, \ell\,.
 \ee
 with $\alpha>0$, $\beta>1$ fixed such that
 \be
 \label{alphabeta}
\alpha^2+ (3\, \beta/4)^2<1.
 \ee
By construction it holds
 \[
 \B_{\K_{t, \mu_t, \omega_t^\bot}}(\omega_t^\bot)=\frac{1}{4}\, \B_{\K_t}(\omega_t^\bot)\,,
 \]
 hence, by \eqref{condt}, $\K_t\setminus \K_{t,\mu_t, \omega_t^\bot}$ is contained in a rectangle having edges of length at most
\[
\B_{\K_t}(\omega_t)\le\alpha\, \ell\qquad  \text{and} \qquad   \frac{3}{4}\, \B_{\K_t}(\omega_t^\bot)\le \frac{3}{4}\, \beta\, \ell\,.
\]
Condition \eqref{alphabeta} ensures that ${\rm diam}(\K_t\setminus\K_{t, \mu_t, \omega_t^\bot})<\ell$, hence 
\be
\label{inter}
 \A\cap  \K_{t, \mu_t, \omega_t^\bot}\quad \text{is a segment of positive length for all $t\in (\bar t, M)$}.
 \ee
  
We will reach a contradiction   in each of the following three cases:  
\begin{enumerate}
\item For some $t\in (\bar t, M)$, $\omega_t^\bot\ne \pm \omega_\A$, the direction of $\A$;
\item For all $t\in (\bar t, M)$, $\omega_t^\bot\parallel\omega_\A$, but for some $t\in (\bar t, M)$ it holds
\be
\label{case2}
\liminf_{\omega\to \omega_t^\bot}\F_{\K_{t, \lambda_t, \omega_t^\bot}}(\omega)\ge \F_{\K_{t, \lambda_t, \omega_t^\bot}}(\omega_t^\bot)\,;
\ee
\item For all $t\in (\bar t, M)$, $\omega_t^\bot\parallel\omega_\A$ and
\be
\label{case3}
\liminf_{\omega\to \omega_t^\bot}\F_{\K_{t, \lambda_t, \omega_t^\bot}}(\omega)< \F_{\K_{t, \lambda_t, \omega_t^\bot}}(\omega_t^\bot)\,.
\ee
\end{enumerate}

\medskip
In case (1), we consider, for the corresponding $t$, the solution $ u_t$ of \eqref{equation} given by   
\[
 u_t(x)=u(T_{\mu_t, \omega_t^\bot}(x))\,,
\]
which, by \eqref{incl}, is well defined and positive  in $\K_{t, \mu_t, \omega^\bot_t}$ and fulfils
\[
u_t\ge u>0\qquad \text{on }\ \partial\K_{t, \mu_t, \omega^\bot_t}\,.
\]
The weak comparison principle of Lemma \ref{wc} implies that $u_t\ge u$ on $\K_{t, \mu_t, \omega^\bot_t}$, which is impossible since $u=M$ on the segment $\A\cap \K_{t, \mu_t, \omega^\bot_t}$, while  $\{x\in \K_{t, \mu_t, \omega^\bot_t}: u_t=M\} $
is a segment having direction $\omega_t^\bot\ne\pm\omega_\A$.

Consider case (2). Relations \eqref{mut} and \eqref{incl} imply that $\K_{t, \mu_t, \omega_t^\bot}$ is a cap of $ \K_{t, \lambda_t, \omega_t^\bot}$ in direction $\omega_t^{\bot}$ such that
\[
T_{\mu_t, \omega_t^\bot}(\K_{t, \mu_t, \omega_t^\bot})\subseteq \K_{t, \lambda_t, \omega_t^\bot}\,,
\]
therefore
\[
\F_{\K_{t, \lambda_t, \omega_t^\bot}}(\omega_t^\bot)\ge \frac{1}{4}\, \B_{\K_t}(\omega_t^\bot)\,.
\]
Thus \eqref{case2} forces, for some $t\in (\bar t, M)$,
\[
\liminf_{\omega\to \omega_t^\bot}\F_{\K_{t}}(\omega)\ge\liminf_{\omega\to \omega_t^\bot}\F_{\K_{t, \lambda_t, \omega_t^\bot}}(\omega)\ge \frac{1}{4}\, \B_{\K_t}(\omega_t^\bot)\, ,
\]
which allows to select a direction $\tilde\omega_t\ne \pm\omega_\A$ such that \eqref{incl} and \eqref{inter} continue to hold with $\tilde{\omega}_t$ instead of $\omega_t$. These are the only conditions needed to run the argument of case (1), giving again a contradiction.

 It remains to consider case (3), where in particular $\omega_t^\bot\parallel\omega_\A$ for all $t\in (\bar t, M)$.  Condition \eqref{case3} allows to apply Lemma \ref{rectangle}, hence for all $t\in (\bar t, M)$  the sections $\pi_{\lambda, \omega_\A}\cap \K_t$ are shadows of $\K_t$ in direction $\omega_\A$ for any $\lambda$ obeying   
 \[
\lambda_t\le \lambda \le  \frac{\lambda_t+\H_{\K_t}(\omega_\A)}{2}\,.
\]
  Since 
\[
 \frac{\lambda_t+\H_{\K_t}(\omega_\A)}{2}\ge \bar\lambda_t=\frac{\H_{\K_t}(\omega_\A)-\H_{\K_t}(-\omega_\A)}{2}\, ,
 \]
 we can suppose from the beginning that the $\lambda_t$ found in \eqref{barlambda} coincides with $\bar\lambda_t$.
 In particular, \eqref{condlambda}  holds in both directions $\omega_t$ and $-\omega_t$. Clearly  case (1) does not hold and, by the previous argument, we can rule out case (2) in both directions $\pm\omega_t$. So case (3) must hold, with \eqref{case3} being fulfilled in both directions $\pm\omega_t$. 
 
 Applying again Lemma \ref{rectangle} we infer that  the set $\K_t\cap S_t$ is a rectangle with sides parallel to $\omega_\A$ and $\omega_\A^\bot$, where 
\[
S_t= \Big\{x\in \R^2:\bar\lambda_t-\frac{1}{4}\, \B_{\K_t}(\omega_\A)\le \langle x, \omega_\A\rangle\le \bar\lambda_t+\frac{1}{4}\, \B_{\K_t}(\omega_\A)\Big\}.
\]
 Since $\K_t\to \A$ in Hausdorff distance, recalling that $x_0$ is the midpoint of $\A$ and $\ell$ its length, we have
 \[
  \bar\lambda_t\pm\frac{1}{4}\, \B_{\K_t}(\omega_\A)\to \langle x_0, \omega_\A\rangle\pm\frac{\ell}{4}
 \]
 as $t\to M$. Therefore, for a sufficiently small $\delta>0$,  the strip
 \[
 S=\bigcap_{M-\delta\le t\le M}S_t
 \]
 has positive width (almost $\ell/2$) and the level sets of $\{x\in S\cap \Omega: u(x)=t\}$ are all parallel to $\omega_\A$ for all $t\in (M-\delta, M)$. Thus  
 \[
 u(x)=h(\langle x, \omega_\A^\bot\rangle) \qquad \text{in }\  S\cap \K_{M-\delta}
 \]
 for some $h\in C^1(\R, \R_+)$. 
 
 The function $h$  solves the one-dimensional version of equation \eqref{equation} and it is readily checked that $h$ is even, positive and $h'$ vanishes only at the maximum point. Since $\nabla u=0$ on $\A$, it follows   that $\K_{M-\delta}\cap \pi_{\bar\lambda_t, \omega_\A^\bot}$ is a segment of length ${\rm width}\, (\K_{M-\delta}/2)$ with midpoint at $x_0$. Such a segment is also a shadow of $\K_{M-\delta}$, hence
 \be
 \label{hstrip}
 \begin{split}
 \K_{M-\delta}&\subset \left\{x\in \R^2: -{\rm width}\, (\K_{M-\delta})/2\le\langle x-x_0, \omega_\A^\bot\rangle\le {\rm width}\, (\K_{M-\delta})/2\right\}\\
 &=\{x\in \R^2:h(\langle x, \omega_\A^\bot)\ge t-\delta\}.
 \end{split}
 \ee
  We conclude by the strong comparison principle of Proposition \ref{sc}: the function  
  \[
  v(x)= h(\langle x, \omega_\A^\bot\rangle)
  \]
  is a positive solution of \eqref{equation} in $\K_{M-\delta}$, its gradient vanishes only on the line through $\A$ and   $v=u$ in an open subset of $\K_{M-\delta}$. Moreover,  \eqref{hstrip} implies that $v\ge u>0$ on $\partial \K_{M-\delta}$, thus the weak comparison principle of Lemma \ref{wc} implies that $v\ge u$. Since $\A=\{x\in \K_{M-\delta}: \nabla v(x)=\nabla u(x)=0\}$ and   $\K_{M-\delta}\setminus \A$ is connected, $u$ and $v$ must coincide.  Thus $u$ is one-dimensional in the whole $\K_{M-\delta}$, contradicting its boundedness and concluding the proof.
  \end{proof}

 We now turn to the proofs of Corollary \ref{scvarphi},   while the proof of \ref{sclog} follows similarly and is omitted. As remarked in the introduction, the following proof holds in all cases where $f$ obeys the assumptions of Corollary \ref{scvarphi} (or, with the same proof, \ref{sclog}), whenever $\Omega$ is a bounded convex domain such that problem \eqref{equation} has a unique solution. 
 
 \begin{proof}[Proof of Corollary \ref{scvarphi}]
 By \cite[Theorem 1.1]{BMS} we know that if $u\in W^{1,p}_0(\Omega)$ solves \eqref{equation}, then  $ v:=\varphi\circ  u$
 is concave, and thus $u$ is quasi-concave.  Moreover, well-known regularity arguments ensure $u$ and $v$ are of class $C^1(\Omega)$. By Theorem \ref{Mth},  $u$ (and thus $v$) attains its maximum  at a single point $\overline{x}\in\Omega$.  By the concavity of $v$, its gradient vanishes only at the maximum points, thus $\bar x$ is actually the only critical point of $v$.

 As computed in \cite{BMS}, the equation solved by $v$ in $\Omega$ is  
 \begin{equation}\label{eq:veq}
 -\Delta_p v=\frac{\psi''(v)}{\psi'(v)}\, (p+(p-1)\,\vert\nabla v\vert^p)\,, 
 \end{equation}
 where $\psi=\varphi^{-1}$ and 
 \[
 \frac{\psi''}{\psi'}=F^{1/p}\circ \psi \, \in C^{2, \alpha}\,. 
 \]
 Standard regularity theory applies in $\Omega\setminus\{\bar x\}$ ensuring that $v\in C^{3, \alpha}(\Omega\setminus\{\bar x\})$.

In terms of the convex function $w:=-v$, equation \eqref{eq:veq} in $\Omega\setminus\{\bar x\}$ can be written as
\[
G(D^2 w, Dw, w)=0\,,
\]
 where $G\in C^3\big(\mathcal S_+\times (\R^2\setminus\{0\})\times -\varphi(\R_+)\big)$ (recall that ${\mathcal S}_+$ is the cone of positive definite $2\times 2$ matrices) is of the form
\[
 G(X,\xi, t)=-\frac{1}{\Tr (A(\xi)\, X)}+ \frac{\psi'(-t)}{\psi''(-t)}\, b(\xi)\,.
\]
with
 \[
A(\xi):=I+(p-2)\, \frac{\xi}{|\xi|}\otimes\frac{\xi}{|\xi|}\,,
\]
which is positive definite for all $\xi\ne 0$, $p>1$,  and
 \[
 b(\xi):=\frac{1}{\vert \xi\vert^{2-p}\, (p+(p-1)\, \vert\xi\vert^p)}\,,
 \]
 which is non-negative.
We thus apply the microscopic convexity principle \cite[Theorem 1.1]{bianguan}, later improved in \cite{SW}. It is readily checked that, in $\Omega\setminus\{\overline{x}\}$, it holds 
\[
\lambda(x)\, |z|^2\le \langle D_XG(D^2 w(x), Dw(x), w(x))\, z, z\rangle\le \Lambda(x)\, |z|^2\qquad \forall z\in \R^2\,.
\]
for some  $\lambda, \Lambda\in C^0(\Omega\setminus\{\overline{x}\}, \R_+)$. As proved in \cite{BMS}, the assumptions on $F$ imply that the function $t\mapsto \psi'(-t)/\psi''(-t)$ is convex. Finally, the Appendix of \cite{ALL} shows that
\[
X\mapsto -\frac{1}{\Tr (A \, X^{-1})}
\]
is convex on ${\mathcal S}_+$, for each fixed $\xi\ne 0$ and $A$ positive definite, so that the map
 \[
 (X, t)\mapsto G(X^{-1},\xi, t) 
 \]
 is convex on ${\mathcal S}_+\times -\varphi(\R_+)$ for any fixed $\xi\ne 0$. Thus by \cite[Theorem 1.1]{SW} we conclude that the Hessian of $w$ has constant rank in $\Omega\setminus\{\overline x\}$. 
 
 We claim that  $D^2w$, or equivalently $D^2v$, has full rank on $\Omega\setminus\{\overline x\}$.
 Arguing by contradiction, suppose that ${\rm det}\, D^2v\equiv 0$   in $\Omega\setminus\{\overline x\}$.  This implies (see \cite{HN} or \cite[Theorem 2]{JP} for a more modern statement) that the graph of $v$ over $\Omega\setminus\{\overline x\}$ is developable:  for any point $x\in\Omega\setminus\{\overline x\}$, either $Dv$ is locally constant near $x$, or there is a line $l_x$ through $x$ such that $D v$ is constant on the connected component of $l_x\cap\big(\Omega\setminus\{\overline x\}\big)$ containing $x$.
The first alternative cannot hold, since otherwise the left hand side of \eqref{eq:veq} would vanish on an open set, while its  right hand side is strictly positive. So, we are left with the second alternative. 

Let $M=\sup_\Omega v$ and fix 
\[
 m\in \big(\inf_\Omega v, M\big).
 \]
 Given $\eps>0$, we  choose a point $x_0\in \Omega\setminus\{\bar x\}$ such that 
\be
\label{asd}
|Dv(x_0)|<\eps, \qquad v(x_0)>\frac{M+m}{2}\, .
\ee
The connected component of $l_{x_0}\cap \big(\Omega\setminus\{\bar x\}\big)$ containing $x_0$, provided by the second alternative of the developability of $v$, must intersect   $\partial\{x\in\Omega: v(x)> m\}$ at some point $x_1$ where $v(x_1)=m$.  Moreover, it holds 
\[
Dv(t\, x_0+(1-t)\, x_1)\equiv Dv(x_0)\qquad \text{for all $t\in [0, 1]$}.
\]
But then from \eqref{asd} we obtain
\[
\frac{M-m}{2}< v(x_0)-v(x_1) =  \int^1_0\frac{d}{dt}v(t\, x_0+(1-t)\, x_1)\, dt\le |Dv(x_0)|\, |x_0-x_1|< \eps\, {\rm diam}\, \Omega\, ,
\]
and taking $\eps$ sufficiently small gives a contradiction. Therefore $D^2 v$ has full rank in $\Omega\setminus\{\bar x\}$ and is positive definite there. 

It remains to prove that $v$ is strictly concave in $\Omega$, a property that can be characterised by strict concavity on each segment $[x, y]\subseteq\Omega$. For any such segment $[x, y]$, consider the function $g(t)=v(t\, x+(1-t)\, y)$: it is readily checked that $g\in C^1$  with strictly decreasing derivative, whether or not $\bar x\in [x, y]$. Thus $g$ is strictly concave and so is $v$.
   \end{proof}

 \appendix

\section{An example in three dimensions}
The next lemma shows that \eqref{eqlem1} cannot hold for arbitrary convex bodies in dimension 3.
\begin{lemma}\label{lemmaApp}
For $\alpha>0$ let $\K^\alpha\subseteq \R^3$ be defined as
\[
\K^\alpha=\big\{(x_1, x_2, x_3)\in \R^3: |x_3|\le \alpha \, x_1,\  |x_2|\le \alpha\, (1-x_1)\big\}
\]
(see Figure 5). Then for any $\omega\in \S^{2}$ it holds
\[
\F_{\K^\alpha}(\omega)\le 2\,  \alpha\, .
\]
\end{lemma}
\begin{figure}[h]
\centering
\begin{tikzpicture}[line cap=round,line join=round,x=1.3cm,y=1cm]
\draw[-latex] (0, 0)--(0, 3.5) node[left]{$x_3$};
\draw (0, -2.5)--(0, -1.11);
\draw[dotted, gray] (0, -1.1)--(0, 0);
\draw[-latex] (-1, -2)--(1, 2) node[right]{$x_2$};
\draw (-0.5, -1) node[above left]{$z_2$}--(0.5, 1) node[above left]{$z_1$}--(5, 0.6) node[above]{$z_3$}--(5, -2.1) node[below]{$z_4$}--(-0.5, -1);
\shade[top color=white!80!black, opacity=0.4, shading angle=90] (-0.5, -1)--(0.5, 1)--(5, 0.6)--(-0.5, -1);
\shade[top color=white!50!black, opacity=0.4, shading angle=180] (-0.5, -1)--(5, 0.6)--(5, -2.1);
\draw (5, 0.6)--(-0.5, -1);
\draw[dotted, gray] (0, 0)--(6, -1);
\draw[dotted, gray] (0.5, 1)--(5, -2.1);
\draw[very thick, gray, opacity=0.7] (1.2, -0.2)--(3.72, -0.62);
\draw [-latex] (5, -0.835)--(6, -1) node[below]{$x_1$};
\filldraw (-0.5, -1) circle (1pt) (0.5, 1) circle (1pt) (5, 0.6) circle (1pt) (5, -2.1) circle (1pt);
\draw [decoration = {calligraphic brace}, decorate] (5.1, 0.6) node{} -- node[midway, right]{$\alpha$} (5.1, -0.835);
\end{tikzpicture}

\caption{The convex body $\K^\alpha$ and, inside it, its heart.}
\end{figure}
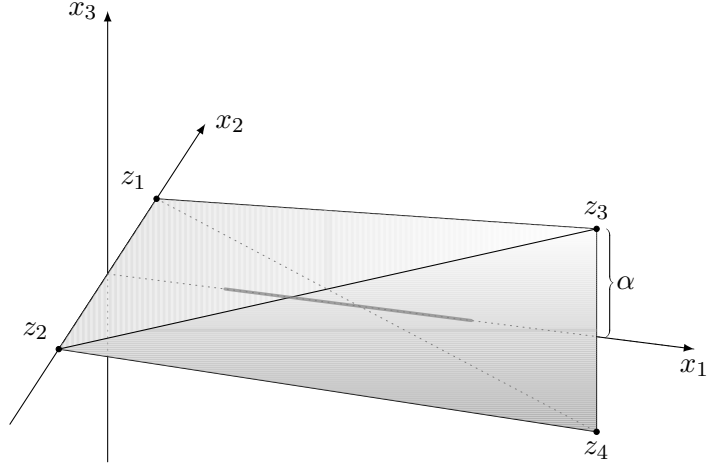

\begin{proof}
It is readily checked that 
\[
\K^\alpha={\rm co}\, \big(\{z_1, z_2, z_3, z_3\}\big)\, ,
\]
where
\[
z_1=(0, \alpha, 0)\, , \quad z_2=(0, -\alpha, 0)\, , \quad z_3=(1, 0, \alpha)\, , \quad z_4=(1, 0, -\alpha)\, \,,
\]
and ${\rm co}$ denotes the convex hull of a set.

In particular, $\K^\alpha$ has two trivial symmetries with respect to the planes $x_3=0$ and $x_2=0$, and a third symmetry given by 
\[
(x_1, x_2, x_3)\mapsto (1-x_1, x_3, x_2)\, .
\]
The first two symmetries exchange $z_3$ with $z_4$ and $z_1$ with $z_2$ respectively, while the third one exchanges  $z_1$ with $z_3$ and $z_2$ with $z_4$. 
In particular, the symmetries of $\K^\alpha$ induce the full permutation group on its extremal points.

Given $\omega=(\omega_1,\omega_2,\omega_3)\in \S^{2}$, let $\pi_{\lambda, \omega}$ be a plane intersecting $\K^\alpha$ such that $T_{\lambda, \omega}(\K^\alpha_{\lambda, \omega})\subseteq \K^\alpha$ and 
\[
\F_{\K^\alpha}(\omega)=\B_{\K^\alpha_{\lambda, \omega}}(\omega)\, .
\]
 We can suppose, without loss of generality due to the symmetries of $\K^\alpha$, that  $z_1\in \K^\alpha_{\lambda, \omega}$, i.\,e.
\be
\label{z1}
\langle z_1, \omega\rangle>\lambda\, ,
\ee
while, recalling \eqref{reflection}, $T_{\lambda, \omega}(z_1)\in \K^\alpha$ if and only if
\be
\label{Tz1}
\begin{cases}
|\omega_3|\le -\alpha\, \omega_1\\[2pt]
|\alpha-2\, \omega_2\, (\langle z_1, \omega\rangle-\lambda)|\le \alpha\, \big(1+2\, \omega_1\, (\langle z_1, \omega\rangle-\lambda)\big).
\end{cases}
\ee
These conditions imply
\be
\label{tz1}
|\omega_3|\le -\alpha\, \omega_1\le \omega_2
\ee
and in particular $\omega_1\le 0$.
Suppose   that 
\[
\langle z_2, \omega\rangle>\lambda\, .
\]
From  $T_{\lambda, \omega}(z_2)\in \K^\alpha$, with a similar computation as before,  we get in particular that $\omega_2\le \alpha\, \omega_1$.
 Recalling \eqref{tz1} we thus have
\[
|\omega_3|\le -\alpha\, \omega_1\le \omega_2\le \alpha\, \omega_1
\]
which forces $\omega=(0, 0, 0)$, giving a contradiction and proving that $\langle z_2, \omega\rangle\le \lambda$, which is equivalent to
\be
\label{z2}
 -\lambda\le \langle z_1, \omega\rangle\, .
\ee
Next we claim that
\be
\label{cls}
\sup\big\{\langle z, \omega\rangle: z\in \K^\alpha\big\}=\langle z_1, \omega\rangle\, .
\ee
Indeed, the supremum is attained at an extremal point for $\K^\alpha$, so it suffices to evaluate $\langle z_i, \omega\rangle$ for $i=1,\dots, 4$. Clearly \eqref{z1} and \eqref{z2} ensure that $\langle z_2, \omega\rangle \le \langle z_1, \omega\rangle$, while \eqref{tz1}  implies, in addition to $\omega_1\le 0$,
\[
\langle z_3, \omega\rangle= \omega_1+\alpha\, \omega_3\le \alpha\, |\omega_3|\le \alpha\, \omega_2=\langle z_1, \omega\rangle
\]
and similarly
\[
\langle z_4, \omega\rangle= \omega_1-\alpha\, \omega_3\le \alpha \, |\omega_3|\le \langle z_1, \omega\rangle\, .
\]
Thus \eqref{cls} is proved, showing that the plane with normal vector $\omega$ passing through $z_1$ is a support plane for $\K^\alpha$. It follows that 
\[
\B_{\K^\alpha_{\lambda, \omega}}(\omega)=\langle z_1, \omega\rangle-\lambda
\]
and we can finally take advantage of \eqref{z2} to get
\[
\B_{\K_{\lambda, \omega}^\alpha}(\omega)\le 2\, \langle z_1, \omega\rangle=2\, \alpha\, \omega_2\le 2\, \alpha\, ,
\]
where we have used the fact that $\omega_2\geq0$, by \eqref{tz1}.

\end{proof}

The failure of \eqref{eqlem1} does not imply in itself that the plan outlined in the Introduction must fail in dimension 3. However, it does so, as the following refinement of Lemma \ref{lemmaApp} shows.
Recall that the {\em heart} of a convex body is defined as
\[
\heartsuit\, (\K)= \K\setminus \bigcup\big\{\K_{\lambda, \omega}:T_{\lambda, \omega}(\K_{\lambda, \omega})\subseteq \K\big\}\,.
\]
\begin{lemma}
For $\alpha>0$, let $\K^\alpha$ be as in the previous Lemma.  Then
\be
\label{fin}
\heartsuit\,  (\K^\alpha)\supseteq \big\{(t, 0, 0): 2\, \alpha\le t\le 1-2\, \alpha\big\}. 
\ee
for all sufficiently small $\alpha>0$.
 \end{lemma}

\begin{proof}
Define
\[
\K^0=\big\{(t, 0, 0)\in \R^3: 0\le t\le 1\big\}.
\]
By \cite[Theorem 2.4]{BM} and the symmetries of $\K^\alpha$, we already know that 
\be
\label{bm}
\heartsuit\,  (\K^\alpha)\subseteq \K^0\, .
\ee
Let $\pi_{\lambda, \omega}$ be as in the previous proof, i.\,e., fulfilling
\[
T_{\lambda, \omega}(\K^\alpha_{\lambda, \omega})\subseteq \K^\alpha, \qquad \langle z_1, \omega\rangle>\lambda\, .
\]
 We claim that  for sufficiently small $\alpha$'s it holds
\be
\label{clst}
\sup\big\{t: (t, 0, 0)\in \K_{\lambda, \omega}^\alpha\big\}\le 2\, \alpha\, .
\ee
Before proving the claim, let us show how \eqref{clst} (under the assumption $z_1\in \K_{\lambda, \omega}^\alpha$) implies \eqref{fin}. By the symmetries of $\K^\alpha$ discussed at the beginning of the proof of Lemma \ref{lemmaApp}, we  infer that \eqref{clst} holds true if  $z_2\in \K_{\lambda, \omega}^\alpha$, while  if $ z_3$ or $z_4$ belong to $\K_{\lambda, \omega}^\alpha$, we get
\[
\inf\big\{t: (t, 0, 0)\in \K_{\lambda, \omega}^\alpha\big\}\ge 1- 2\, \alpha\, .
\]
 These two inequalities show that in all cases when $T_{\lambda, \omega}(\K_{\lambda, \omega}^\alpha)\subseteq \K^\alpha$, the cap $\K_{\lambda, \omega}^\alpha$ cuts a segment of length at most $2\, \alpha$ either on the left or on the right of $\K^0$. Recalling \eqref{bm}, this shows \eqref{fin}.

We next focus on proving \eqref{clst}, assuming that
\be
\label{assA}
\sup\big\{t: (t, 0, 0)\in \K_{\lambda, \omega}^\alpha\big\}>0
\ee
(otherwise there is nothing to prove). From the previous proof, we know \eqref{Tz1} is fulfilled and is equivalent to
\be
\label{Tz1bis}
\begin{cases}
|\omega_3|\le -\alpha\, \omega_1\le \omega_2\\[2pt]
\alpha\ge (\omega_2-\alpha\, \omega_1) (\langle z_1, \omega\rangle-\lambda)\, ,
\end{cases}
\ee
 meanwhile \eqref{z2} holds true as well.
Notice that it always holds $\omega_1\le 0\le \omega_2$ and that $\omega_1=0$ forces $\omega=(0, 1, 0)$. In this case $\K_{\lambda, \omega}^\alpha$ does not intersect the $x_1$-axis and the supremum in \eqref{assA} is $-\infty$, so we can suppose that $\omega_1<0$.

The maximum in \eqref{clst} is attained for some $\bar t$ such that $\bar z=(\bar t, 0, 0)\in \K^0$ and
\[
\bar z\in \partial \K_{\lambda, \omega}^\alpha\subseteq\partial \K^\alpha\cup \pi_{\lambda, \omega}\, .
 \]
The only  point in $\K^0\cap \partial \K^\alpha$ fulfilling \eqref{assA} is $\bar z=(1, 0, 0)$, which cannot lie in $\K_{\lambda, \omega}^\alpha$: otherwise, from $T_{\lambda, \omega}(\bar z)\in \K^\alpha$ we would infer\footnote{from the condition $|x_2|\le \alpha\, (1-x_1)$ evaluated for $T_{\lambda, \omega}(\bar z)$ and cancelling out the positive factor $(\langle \bar z, \omega\rangle-\lambda)$} $\omega_2 \le \alpha \, \omega_1$, which implies $\omega_1=\omega_2=0$ (by the first condition in \eqref{Tz1bis}), but then we would have  $\langle \bar z, \omega\rangle=0$, against $\bar z\in\K_{\lambda, \omega}^\alpha$.  We therefore conclude that $\bar z\in \pi_{\lambda, \omega}$ and thus (recall that $\omega_1<0$ and \eqref{assA} is assumed) we have
\[
\bar t=\sup\big\{t: (t, 0, 0)\in \K_{\lambda, \omega}^\alpha\big\}=\lambda/\omega_1>0\, ,
\]
and, in particular, $\lambda<0$.
Now, on one hand \eqref{z2}  provides the bound
\be
\label{b1}
\bar t\le \alpha\, \frac{\omega_2}{- \omega_1}\, .
\ee
On the other hand the second condition \eqref{Tz1bis} reads
\[
-\lambda\le \frac{\alpha}{\omega_2-\alpha\, \omega_1}-\alpha\, \omega_2
\]
which, divided by  $-\omega_1>0$, is equivalent to 
\[
\frac{\lambda}{\omega_1}\le \alpha\, \frac{1-\omega_2^2+\alpha\, \omega_1\, \omega_2}{\alpha\, \omega_1^2-\omega_1\, \omega_2}\, .
\]
Using the first conditions in \eqref{Tz1bis} and $|\omega|=1$, we get
\[
1-\omega_2^2=\omega_1^2+\omega_3^2\le (1+\alpha^2)\, \omega_1^2
\]
which, inserted in the previous display, gives the second bound
\be
\label{b2}
\bar t= \frac{\lambda}{\omega_1}\le \alpha\, \frac{-(1+\alpha^2)\, \omega_1-\alpha\, \omega_2}{\omega_2-\alpha\, \omega_1}\, .
\ee
In terms of the auxiliary non-negative variable $\xi=\omega_2/(-\omega_1)$, the bounds \eqref{b1} and \eqref{b2} yield
\[
\bar t\le\alpha\,  \max_{\xi\ge 0} \min\Big\{\xi, \frac{1+\alpha^2-\alpha\, \xi}{\xi+\alpha}\Big\}
\]
which is readily evaluated as 
\[
\bar t\le \alpha\, \Big(\sqrt{2\, \alpha^2+1}-\alpha\Big),
\]
proving \eqref{clst} for sufficiently small $\alpha$.
\end{proof}

Thanks to the previous Lemma, we can provide the example mentioned in the Introduction.

\begin{ex}\label{exapp}
We can rescale the convex body $\K^\alpha$ by a multiple $l>0$, to obtain the general form of \eqref{fin} for sufficiently small $\alpha>0$, i.\,e.
\be
\label{fin2}
\heartsuit\,  ( l\, \K^\alpha)\supseteq \big\{(t, 0, 0): 2\, \alpha\, l\le t\le l\, (1-2\, \alpha)\big\}.
\ee
For $l_n=1+1/n$, we choose $\alpha_n>0$ sufficiently small such that 
\be
\label{fin3}
 l_n\, (1-4\, \alpha_n\, l_n)\ge 1
\ee
and in particular $\alpha_n\to 0$ as $n\to +\infty$. Then we define the sequence of convex bodies
\[
\K_n=(-2\, \alpha_n\, l_n, 0, 0)+l_n\, \K^\alpha_n\, .
\]
By \eqref{fin2} and \eqref{fin3}, it holds 
\[
\heartsuit\, (\K_n)\supseteq \K^0=\big\{(t, 0, 0): 0\le t\le 1\big\},
\]
while $\K_n\to \K^0$ in the Hausdorff metric. By the very definition of $\heartsuit \K_n$, all foldable caps of $\K_n$ are disjoint from $\K^0$.
\end{ex}


\begin{thebibliography}{10}

\bibitem{APP}
{\sc A. Acker, L.~E. Payne, G. Philippin},  {\em On the convexity of level lines of the fundamental mode in the clamped membrane problem, and the existence of convex solutions in a related free boundary problem},  Z. Angew. Math. Phys. {\bf 32} (1981), 683--694.

\bibitem{ALL}
{\sc O. Alvarez, J.-M. Lasry, and P.-L. Lions}, {\em Convex viscosity solutions
  and state constraints}, J. Math. Pures Appl. {\bf 76} (1997), 265--288.



\bibitem{bianguan}
{\sc B. Bian and P. Guan}, {\em A microscopic convexity principle for nonlinear
  partial differential equations}, Invent. Math. {\bf 177} (2009), 307--335.



\bibitem{BF}
{\sc T. Bonnesen and W. Fenchel},  {\em Theory of convex bodies}.  L. Boron, C. Christenson and B. Smith eds.,  BCS Associates, Moscow, Idaho, USA, 1987.

\bibitem{BMS}
{\sc W. Borrelli, S. Mosconi and M. Squassina},  {\em Concavity properties for solutions to p-Laplace equations with concave nonlinearities}.  preprint,  arXiv:2111.14801, 2021.

\bibitem{BL}
{\sc H.~J. Brascamp and E.~H. Lieb}, {\em On extensions of the
  {B}runn-{M}inkowski and {P}r\'{e}kopa-{L}eindler theorems, including
  inequalities for log concave functions, and with an application to the
  diffusion equation}, J. Functional Analysis {\bf 22} (1976), 366--389.


  
    \bibitem{BM0}
  {\sc L. Brasco, R. Magnanini and P. Salani}, {\em The location of the hot spot in a grounded convex conductor}, Indiana Univ. Math. J. {\bf 60} (2011), 633--659.


  
\bibitem{BM}
{\sc L. Brasco and R. Magnanini}, {\em The heart of a convex body}. In  Geometric properties for parabolic and elliptic PDE's, 49--66, Springer INdAM Ser., {\bf 2}, Springer, Milan, 2013. 

 \bibitem{BPZ}
  {\sc L. Brasco, F. Prinari and A.~C. Zagati}, {\em A comparison principle for the Lane-Emden equation and applications to geometric estimates},  preprint arXiv:2111.09603v1, 2021.





\bibitem{CaFri}
{\sc L.~A. Caffarelli and A. Friedman}, {\em Convexity of solutions of
  semilinear elliptic equations}, Duke Math. J. {\bf 52} (1985), 431--456.

\bibitem{CS}
{\sc L.~A. Caffarelli and J. Spruck}, {\em Convexity properties of solutions to
  some classical variational problems}, Comm. Partial Differential Equations {\bf 7}
  (1982), 1337--1379.

\bibitem{damascelli}
{\sc  L. Damascelli}, {\em   Comparison theorems for some quasilinear degenerate elliptic operators and applications to symmetry and monotonicity results}. Ann. Inst. H. Poincaré Anal. Non Linéaire {\bf 15} (1998), 493--516.





\bibitem{DS}
{\sc L. Damascelli and B. Sciunzi}, {\em Harnack inequalities, maximum and comparison principles, and regularity of positive solutions of $m$-Laplace equations}, Calc. Var. Partial Differential Equations {\bf 25}  (2006), 139--159.

\bibitem{diaz-saa}
{\sc J.~I. D\'{\i}az and J.~E. Sa\'{a}}, {\em Existence et unicit\'{e} de
  solutions positives pour certaines \'{e}quations elliptiques
  quasilin\'{e}aires}, C. R. Acad. Sci. Paris S\'{e}r. I Math. {\bf 305} (1987),
 521--524.

\bibitem{Gustaffson}
{\sc B. Gustaffson}, {\em On the motion of a vortex in two-dimensional flow of an ideal fluid in simply and
multiply connected domains}, Royal Institute of Technology, Technical Report, 1979, Stockholm, Sweden.

\bibitem{haegi}
{\sc H.~R. Haegi}, {\em Extremalprobleme und Ungleichungen konformer Gebietsgr\"ossen}, Compos. Math, {\bf 8} (1951), 81--111.




\bibitem{HN}
{\sc P. Hartman and L. Nirenberg}, {\em On sperical image maps whose Jacobians do not change sign},
  Amer. J. Math.  {\bf 81} (1959), 901--920.


\bibitem{kaw}
{\sc B. Kawohl}, {\em Rearrangements and convexity of level sets in {PDE}},
   Lecture Notes in Mathematics {\bf 1150}, Springer-Verlag, Berlin, 1985.

\bibitem{kennington}
{\sc A.~U. Kennington}, {\em Power concavity and boundary value problems},
  Indiana Univ. Math. J. {\bf 34} (1985), 687--704.

\bibitem{kore2}
{\sc N. Korevaar}, {\em Capillary surface convexity above convex domains},
  Indiana Univ. Math. J. {\bf 32} (1983), 73--81.

\bibitem{korevaar}
{\sc N.~J. Korevaar}, {\em Convex solutions to nonlinear elliptic and parabolic
  boundary value problems}, Indiana Univ. Math. J. {\bf 32} (1983), 603--614.

\bibitem{kore3}
{\sc N.~J. Korevaar}, {\em Convexity of level sets for solutions to elliptic ring problems}, Comm. Partial Diff. Equations {\bf 15} (1990),  541--556.

\bibitem{korlew}
{\sc N.~J. Korevaar and J.~L. Lewis}, {\em Convex solutions of certain elliptic
  equations have constant rank {H}essians}, Arch. Rational Mech. Anal. {\bf 97}
  (1987), 19--32.


\bibitem{JP}
{\sc R.~L. Jerrard and M.~R. Pakzad}, {\em Sobolev spaces of isometric immersions of arbitrary dimension and co-dimension}, 
Annali di Matematica {\bf 196} (2017), 687--716.


\bibitem{lou}
{\sc H. Lou}, {\em On singular sets of local solutions to $p$-Laplace equations},  Chin. Ann. Math.
{\bf 29B} (2008), 521--530.

\bibitem{Makar}
{\sc L.~G. Makar-Limanov}, {\em The solution of the {D}irichlet problem for the
  equation {$\Delta u=-1$} in a convex region}, Mat. Zametki  {\bf 9} (1971), 89--92.

\bibitem{Saka}
{\sc S. Sakaguchi}, {\em Concavity properties of solutions to some degenerate
  quasilinear elliptic {D}irichlet problems}, Ann. Scuola Norm. Sup. Pisa Cl.
  Sci.  {\bf 14} (1987), 403--421.

\bibitem{SW}
{\sc G. Sz\'{e}kelyhidi and B. Weinkove}, {\em Weak {H}arnack inequalities for eigenvalues and constant rank
              theorems},  Commun. Partial Differ. Equ. {\bf 46} (2021), 1585--1600.


\end{thebibliography}
\end{document}